\documentclass[a4paper,11pt]{article}
\usepackage[T1]{fontenc}
\usepackage{lmodern}
\usepackage[margin=28mm]{geometry}
\usepackage{amsmath,amssymb,amsthm,mathtools}
\usepackage{graphicx,booktabs,microtype}
\usepackage[colorlinks=true,linkcolor=blue,citecolor=blue,urlcolor=blue]{hyperref}
\hypersetup{pdftitle={Semidefinite programming bounds for distance distributions of spherical codes},pdfauthor={Oleg R. Musin}}
\newtheorem{thm}{Theorem}[section]
\newtheorem{prop}[thm]{Proposition}
\newtheorem{cor}[thm]{Corollary}

\theoremstyle{definition}

\theoremstyle{remark}

\newtheorem*{aside}{Remark}
\numberwithin{equation}{section}
\newcommand{\R}{\mathbb R}
\newcommand{\Sph}{\mathbb S}
\newcommand{\PSD}{\succeq0}
\newcommand{\ip}[2]{\langle #1,#2\rangle}
\DeclareMathOperator{\rank}{rank}
\DeclareMathOperator{\tr}{tr}
\title{Semidefinite programming bounds for distance distributions of spherical codes}
\author{Oleg R. Musin}
\date{}
\begin{document}
\maketitle

\begin{abstract}
Multipoint semidefinite programming bounds for spherical codes use
polynomials in several variables. We show how to obtain bounds for the
distance distribution using a polynomial of one variable on a prescribed
interval. The multivariate SDP supplies the required positivity; the sign
and size of the univariate polynomial reveal the geometric consequences. We prove
three- and four-point reductions and explain how to choose the polynomial
to bound the number of pairs in a prescribed interval. Three-point experiments
in dimension four illustrate this procedure, including nearly antipodal
pairs, intervals near inner product $1/2$, and a local approach to the kissing
number bound. As a further application, we extract a degree-six polynomial
from the exact four-point construction of de Laat, Leijenhorst, and de Muinck
Keizer. Its sign and four zeros give a direct description of the distance
restriction underlying uniqueness of the $24$-cell.
\end{abstract}

\smallskip


\section{Introduction}

Multipoint semidefinite programming (SDP) bounds for spherical codes
use polynomials in $\binom{k}{2}=k(k-1)/2$ variables at the $k$-point
level, one for each pairwise inner product among the $k$ points.
Thus three-point SDP uses three variables and four-point SDP uses six.
These variables describe the Gram matrix of the points and are subject
to its feasibility conditions.
Positive semidefinite kernel conditions yield bounds for spherical codes,
but the geometric information in a computed solution can be difficult
to read. Our aim is to express part of this information through a
polynomial of one variable on a prescribed interval.

For a code $C\subset\Sph^{n-1}$ with $N$ points, let
\[
 A_t(C)=\frac1N|\{(x,y)\in C^2:x\cdot y=t\}|,\qquad
 R_g(C)=\sum_{t<1}A_t(C)g(t).
\]
Thus $A_1=1$ and $\sum_{t<1}A_t=N-1$. A lower bound for $R_g$, together
with upper bounds for $g$ on subintervals, gives information about the
distance distribution. In particular, a polynomial positive only on a
chosen interval can force pairs to occur in that interval. A polynomial
nonpositive throughout the admissible interval can exclude all distances
except its zeros when the lower bound is zero.

The classical two-point argument requires nonnegative nonconstant
Gegenbauer coefficients and gives $R_g\ge g_0N-g(1)$.
Here the required positivity of sums over the code is supplied by
an auxiliary multivariate polynomial with a positive semidefinite
kernel expansion.
The resulting weight $g$ may have negative Gegenbauer coefficients;
even when all its coefficients are positive, the multipoint estimate
may be much stronger. The auxiliary polynomial is still needed to prove
the bound, but the subsequent geometric argument involves only $g$.

The three-point reduction uses the
Bachoc--Vallentin kernels \cite{bac08a,bac09a,bac09b}; related energy bounds
were developed by Cohn and Woo \cite{cohn12}. We obtain, in a useful
normalization,
\[
 R_g(C)\ge\frac{N-M}{3N}.
\]
The weight and the auxiliary functions can be found simultaneously by
SDP, with upper bounds imposed on the weight over specified intervals.
The case $g=0$ gives a cardinality bound. Varying the intervals gives
bounds for the distance distribution.

This was the purpose of our joint three-point computations with
Dostert, Kolpakov, and Moustrou. We describe the numerical examples
and interval bounds from those unpublished joint materials
\cite{uniq24}. They were sought to constrain the distance
distribution of a $24$-point kissing configuration in dimension four,
at a time when its uniqueness was open. Two degree-$22$ weights target
nearly antipodal pairs; a further search targets the interval
$[0.35,1/2]$. The experiments also suggested combining the three-point
lower bound with a local estimate involving at most four points in a
spherical cap. These examples explain both how the polynomials were
chosen and why their interval behavior matters.

A four-point version follows by separating the five coincidence types
of an ordered quadruple. For kernels indexed by sets of at most two
points, retaining the pair contribution gives another univariate bound.
We illustrate this with the exact four-point solution of de Laat,
Leijenhorst, and de Muinck Keizer \cite{DLLM}, who proved uniqueness of
the optimal kissing configuration in dimension four. The degree-six
polynomial
\[
 g_6(t)=-t^2(1+t)(1-2t)(1+2t)^2
\]
is nonpositive on $[-1,1/2]$, and their construction implies
$R_{g_6}(C)\ge0$ for every admissible $24$-point code. Its four zeros
therefore give all possible inner products. This makes the distance
restriction visible in a small polynomial. The exact four-point
positivity remains the input from \cite{DLLM}; the degree-six reduction
is an application of the method, rather than an independent proof of
that positivity.

Section~\ref{sec:general} proves the three-point reduction.
Section~\ref{sec:experiments} develops the interval search and its
numerical examples. Section~\ref{sec:four} gives the four-point
reduction and the degree-six application. Two short appendices contain
the rounded coefficients and the exact factorization used in that
application. The supplementary scripts reproduce the polynomial
calculations and figures.

\section{The three-point reduction}
\label{sec:general}

\subsection{Codes, energies, and two-point positivity}

Throughout, $n\ge3$, $N\ge2$, and $T\subset[-1,1)$. An
\emph{$(N,n,T)$ spherical code} is a set $C\subset\Sph^{n-1}$ with
$|C|=N$ and
\[
 I(C):=\{x\cdot y:x,y\in C,\ x\ne y\}\subset T.
\]
The notation $(N,n,\theta)$ will mean $T=[-1,\cos\theta]$.
For a function $f$ defined on the required inner products, let
\begin{equation}\label{eq:energy}
 E_f(C)=\sum_{\substack{x,y\in C\\x\ne y}}f(x\cdot y),
 \qquad
 S_f(C)=\sum_{x,y\in C}f(x\cdot y)=Nf(1)+E_f(C).
\end{equation}
All pair sums are over \emph{ordered} pairs. The energy over unordered
pairs is $E_f/2$.

Let $G_j^{(n)}$ be the Gegenbauer polynomials normalized by
$G_j^{(n)}(1)=1$, orthogonal for the weight
$(1-t^2)^{(n-3)/2}$ on $[-1,1]$. Their positive semidefinite property gives
\[
 M_j(C):=\sum_{x,y\in C}G_j^{(n)}(x\cdot y)\ge0.
\]
In particular, if $f=\sum_{j=0}^d f_jG_j^{(n)}$ and $f_j\ge0$ for
$j\ge1$, then
\begin{equation}\label{eq:two-positive}
 S_f(C)\ge f_0N^2.
\end{equation}
Here the constant coefficient $f_0$ may have either sign. If all
coefficients, including $f_0$, are nonnegative, then $S_f(C)\ge0$.

In particular, for a polynomial $g$ with nonnegative nonconstant
Gegenbauer coefficients,
\begin{equation}\label{eq:LP}
 R_g(C)\ge g_0N-g(1).
\end{equation}
The same positivity underlies the two-point energy bound \cite{Yudin}.

\subsection{Three-point positivity}

For a symmetric function $F:[-1,1]^3\to\R$, set
\[
 S_F(C)=\sum_{x,y,z\in C}F(x\cdot y,x\cdot z,y\cdot z).
\]
The symmetrized Bachoc--Vallentin matrices $S_k^n(t,u,v)$ satisfy
\[
 \sum_{x,y,z\in C}S_k^n(x\cdot y,x\cdot z,y\cdot z)\PSD.
\]
We use their usual normalization with $(S_0^n)_{00}=1$; see
\cite{bac08a,bac09b}. The matrix inner product is
$\ip{H}{S}=\tr(H^{\mathsf T}S)$. If
\begin{equation}\label{eq:SDPexp}
 F(t,u,v)=\sum_{k=0}^d\ip{H_k}{S_k^n(t,u,v)},\qquad
 H_k\PSD\ (k>0),\quad H_0-F_0E_{00}\PSD,
\end{equation}
where $E_{00}$ has its only nonzero entry, equal to one, in position
$(0,0)$, then
\begin{equation}\label{eq:three-positive}
 S_F(C)\ge F_0N^3.
\end{equation}
The scalar $F_0$ in this inequality is distinct from the matrix $H_0$.

The Gram domain for three distinct code points is
\begin{equation}\label{eq:gram-domain}
 D_3(T)=\{(t,u,v)\in T^3:1+2tuv-t^2-u^2-v^2\ge0\}.
\end{equation}
Indeed, the matrix
\[
 \begin{pmatrix}1&t&u\\t&1&v\\u&v&1\end{pmatrix}
\]
is positive semidefinite exactly when $t,u,v\in[-1,1]$ and its
determinant is nonnegative. It is then the Gram matrix of three unit
vectors in $\R^3$, and these vectors are distinct when $t,u,v<1$.
In an ambient dimension smaller than three one would additionally impose
the corresponding rank condition. The counting results below do not
depend on the particular SDP representation.

\begin{thm}[General three-point bound]\label{th3p}
Let $C$ be an $(N,n,T)$ spherical code. Suppose $F$ is symmetric,
$f,g:T\to\R$, and
\begin{align}
 F(1,t,t)&\le f(t) &&(t\in T),\label{eq:diag}\\
 F(t,u,v)&\le g(t)+g(u)+g(v) &&((t,u,v)\in D_3(T)).\label{eq:offdiag}
\end{align}
If $S_F(C)\ge F_0N^3$, then
\begin{equation}\label{eq:three-bound}
 NF(1,1,1)+3E_f(C)+3(N-2)E_g(C)\ge F_0N^3.
\end{equation}
\end{thm}
\begin{proof}
Partition the ordered triples in $C^3$ according to their number of
distinct entries. The all-equal triples contribute $NF(1,1,1)$.
Symmetry of $F$ shows that the triples with exactly two distinct entries
contribute
\[
 3\sum_{x\ne y}F(1,x\cdot y,x\cdot y)\le3E_f(C).
\]
For pairwise distinct triples, use \eqref{eq:offdiag}. Each ordered pair
occurs $N-2$ times in each of the three positions, so their contribution
is at most $3(N-2)E_g(C)$. Add the three estimates and use
\eqref{eq:three-positive}.
\end{proof}

\subsection{A normalized univariate bound}

The next corollary is useful when $g$ is required to be nonpositive
outside a given interval. The positivity assumptions concern the
auxiliary functions $a$ and $F$; the coefficients of $g$ may have either sign.

\begin{cor}[A normalized bound]\label{cor:certificate}
Let $a(t)=\sum_{j=0}^d\alpha_jG_j^{(n)}(t)$ with $\alpha_j\ge0$ for
all $j$, and let $F$ have expansion \eqref{eq:SDPexp} with all $H_k\PSD$
and $F_0=0$. Suppose
\begin{align}
 a(1)+F(1,1,1)&\le M-1,\label{eq:cert1}\\
 a(t)+3F(1,t,t)&\le-1+6g(t) &&(t\in T),\label{eq:cert2}\\
 F(t,u,v)&\le g(t)+g(u)+g(v) &&((t,u,v)\in D_3(T)).\label{eq:cert3}
\end{align}
Then every $(N,n,T)$ spherical code satisfies
\begin{equation}\label{eq:normalized}
 R_g(C)\ge L_N(M):=\frac{N-M}{3N}.
\end{equation}
\end{cor}
\begin{proof}
Use $f=2g-(a+1)/3$ in Theorem~\ref{th3p}. Adding $S_a(C)$ gives
\[
 0\le S_F(C)+S_a(C)
 \le N\bigl(F(1,1,1)+a(1)\bigr)-N(N-1)+3NE_g(C).
\]
The first hypothesis bounds the right side by $N(M-N)+3NE_g$.
Since $E_g=NR_g$, division by $3N^2$ proves the result.
\end{proof}

For $g=0$ we obtain $N\le M$. Other choices of $g$ give bounds for
the distance distribution. It is enough to find functions satisfying
the hypotheses; the value of $M$ need not be optimal.

For fixed polynomial degrees, the coefficient identities and the displayed
inequalities are linear in the unknown coefficients of $a,F,g$ and in $M$.
Sufficient conditions for the inequalities can be expressed by
sum-of-squares representations on the Gram domain and the intervals
in question. With degree bounds fixed, these conditions give an SDP.
The polynomial $g$ may be prescribed or chosen as part of the
optimization, subject to upper bounds on specified intervals.
Some nonnegative polynomials may have no such representation within
the chosen degree bounds. For computational developments in three-point
SDP bounds, see \cite{mv10,mac16}.

\subsection{An additional positive semidefinite block}

We can add a positive semidefinite block to the conditions of
Corollary~\ref{cor:certificate}.
\begin{cor}[A bound with a two-by-two block]\label{cor:blockcertificate}
Let $a$ and $F$ satisfy the positivity assumptions of
Corollary~\ref{cor:certificate}, and let $B=(b_{ij})_{i,j=1}^2\PSD$.
Suppose
\begin{align}
 1+b_{11}+a(1)+F(1,1,1)&\le M,\label{eq:block1}\\
 a(t)+3F(1,t,t)+2b_{12}+b_{22}&\le-1+6g(t)
       &&(t\in T),\label{eq:block2}\\
 F(t,u,v)+b_{22}&\le g(t)+g(u)+g(v)
       &&((t,u,v)\in D_3(T)).\label{eq:block3}
\end{align}
Then $R_g(C)\ge(N-M)/(3N)$ for every $(N,n,T)$ code.
\end{cor}
\begin{proof}
Partitioning triples and using the last two inequalities gives
\[
 S_F(C)+S_a(C)+Nq_B(N-1)
 \le N(M-N)+3NE_g(C),
\]
where $q_B(s)=b_{11}+2sb_{12}+s^2b_{22}=(1,s)B(1,s)^{\mathsf T}$.
All three terms on the left are nonnegative. Divide by $3N^2$.
\end{proof}

Setting $B=0$ gives Corollary~\ref{cor:certificate}. If the SDP includes
this block, its entries must also be included in the inequalities above.

\section{Interval bounds and three-point experiments}
\label{sec:experiments}\label{sec:intervals}

\subsection{Choosing the weight for an interval}

For $J\subset T$, put $A(J)=\sum_{t\in J}A_t$. This is the average
degree of the graph joining pairs whose inner product belongs to $J$.
Writing $e(J)$ for its number of edges, we have
\begin{equation}\label{eq:integer}
 A(J)=\frac{2e(J)}N\in\frac2N\mathbb Z_{\ge0},\qquad
 \sum_{t\in T}A_t=N-1,\qquad E_g=NR_g.
\end{equation}
The average $A(J)$ need not be an integer or a bound at each vertex.

\begin{prop}[Two levels of a weight]\label{prop:interval}
Suppose $R_g(C)\ge L$, and $g\le a$ on $J$, $g\le b$ on $T\setminus J$.
If $a>b$, then
\begin{equation}\label{eq:lowerinterval}
 A(J)\ge\frac2N\left\lceil
 \frac N2\frac{L-b(N-1)}{a-b}\right\rceil.
\end{equation}
If $a<b$, then
\begin{equation}\label{eq:upperinterval}
 A(J)\le\frac2N\left\lfloor
 \frac N2\frac{b(N-1)-L}{b-a}\right\rfloor.
\end{equation}
\end{prop}
\begin{proof}
By \eqref{eq:integer},
$L\le R_g\le aA(J)+b(N-1-A(J))$. Solve this inequality for $A(J)$
and use \eqref{eq:integer}.
\end{proof}

When $a>0$ and $g\le0$ on $T\setminus J$, Proposition~\ref{prop:interval}
gives the two useful specializations
\begin{align}
 g\le a\text{ on }J&\quad\Longrightarrow\quad
 A(J)\ge\frac2N\left\lceil\frac{NL}{2a}\right\rceil,\label{eq:zero-lower}\\
 g\le-a\text{ on }J&\quad\Longrightarrow\quad
 A(J)\le\frac2N\left\lfloor-\frac{NL}{2a}\right\rfloor.\label{eq:zero-upper}
\end{align}
We may take $L=(N-M)/(3N)$ from Corollary~\ref{cor:certificate}
or a lower bound from the four-point reduction below.
If the upper bound is less than $2/N$, then $A(J)=0$.

More generally, if $T=J_1\cup\cdots\cup J_r$ is a disjoint partition
and $g\le b_i$ on $J_i$, then
\begin{equation}\label{eq:multiinterval}
 \sum_{i=1}^r b_iA(J_i)\ge L,\qquad
 \sum_{i=1}^r A(J_i)=N-1,\qquad
 A(J_i)\in(2/N)\mathbb Z_{\ge0}.
\end{equation}
Using several polynomials gives a system of inequalities for the
numbers of edges. Each interval endpoint must belong to only one part
of the partition. Proposition~\ref{prop:interval} also allows a nonzero upper bound outside
the chosen interval.

For a lower bound on $A(J)$, fix $\beta>0$, impose $g\le\beta$ on $J$
and $g\le0$ on $T\setminus J$, and minimize $M$ under
\eqref{eq:cert1}--\eqref{eq:cert3} (or their block version).
For an upper bound, impose $g\le-\beta$ on $J$ and $g\le0$ elsewhere.
These are convex searches when the polynomial degrees are fixed.
For an interval $[a,b]$, an upper bound $g\le\beta$ can be enforced by
\[
 \beta-g(t)=\sigma_0(t)+(t-a)(b-t)\sigma_1(t),
\]
where $\sigma_0,\sigma_1$ are sums of squares; an even degree bound
$2d$ permits degrees $2d$ and $2d-2$, respectively. The complement
is handled interval by interval. The triple inequality is imposed on
the Gram domain \eqref{eq:gram-domain}. Thus the interval is an input
to the optimization, not a feature selected after the weight is found.

\subsection{Three computed polynomials}

The computations described here are from joint work with Dostert,
Kolpakov, and Moustrou \cite{uniq24}; the searches were carried out
by Dostert. Here $p_1$ is the degree-$22$ polynomial with
positive Gegenbauer coefficients, $p_2$ is the degree-$22$ polynomial
with mixed signs, and $p_3$ denotes the third plotted weight.
The retained coefficients of the first two are printed in
Appendix~\ref{app:coefficients}. We write $\widehat p_1,\widehat p_2$
for these rounded polynomials, distinguishing them from the
full-precision weights used in the searches.

The two degree-$22$ searches reported
\[
 M_1^{\rm num}=22.645212490128051,\qquad
 M_2^{\rm num}=22.568913678613019.
\]
Table~\ref{tab:targets} evaluates $L_N(M)=(N-M)/(3N)$.
The intended normalization is $p_i(-1)=1/50$.
\begin{table}[htbp]
\centering
\caption{Reported three-point objectives and the resulting numerical
values of $L_N(M)$; these use the unscaled weights.}
\label{tab:targets}
\begin{tabular}{lrrr}
\toprule
Weight & $M^{\rm num}$ & $L_{24}(M^{\rm num})$ & $L_{25}(M^{\rm num})$\\
\midrule
$p_1$ & $22.6452124901$ & $0.0188164932$ & $0.0313971668$\\
$p_2$ & $22.5689136786$ & $0.0198761989$ & $0.0324144843$\\
\bottomrule
\end{tabular}
\end{table}

The first example demonstrates a large improvement over the direct
two-point estimate. All coefficients of $\widehat p_1$ are positive, but
\[
 24c_0-\widehat p_1(1)=0.222\cdot24-57.5715=-52.2435.
\]
The reported three-point value is positive and close to $1/50$.
For $p_2$, the negative nonconstant coefficients prevent the direct
use of \eqref{eq:LP}. Nevertheless, its reported three-point value
is slightly larger. This illustrates two different reasons to use
multipoint positivity for a univariate weight.

The first numerical search aimed at $p_1\le1/50$ on $[-1,-0.73]$
and $p_1\le0$ on the remaining interval. To see the significance of
the objective value, suppose these inequalities and the three-point
conditions hold exactly. For $N=24$, \eqref{eq:zero-lower} gives
\begin{equation}\label{eq:threshold24}
 A(J)\ge\frac1{12}\left\lceil\frac{24-M}{6\beta}\right\rceil,
 \qquad J=[-1,-0.73],\quad\beta=1/50.
\end{equation}
Any $M<22.68$ therefore forces $A(J)\ge1$, or at least twelve
unordered pairs. The reported $M_1^{\rm num}$ is below this threshold. The corresponding
sign interval reported for $p_2$ is $[-1,-1/\sqrt2]$, which is used
in the local calculation below.
The conclusion uses the integer number of pairs, rather than an
assumption that $A(J)$ is an integer.

The third example concerns $J=[0.35,1/2]$ and illustrates the
same search at the opposite end of the admissible interval. Its
original graph is reproduced in Figure~\ref{fig:three}. In a
normalization with upper levels $1$ on $J$ and $0$ outside, a lower
bound greater than $95/12$ is enough to give
\[
 A(J)\ge\frac1{12}\lceil12R_g(C)\rceil\ge8.
\]
Thus one seeks a weight concentrated near $1/2$ and a lower bound
close to $8$. This is the interval interpretation of the third
experiment; we use the graph to illustrate the shape of the weight.

\begin{figure}[htbp]
\centering
\begin{minipage}[b]{.51\textwidth}
\includegraphics[width=\linewidth]{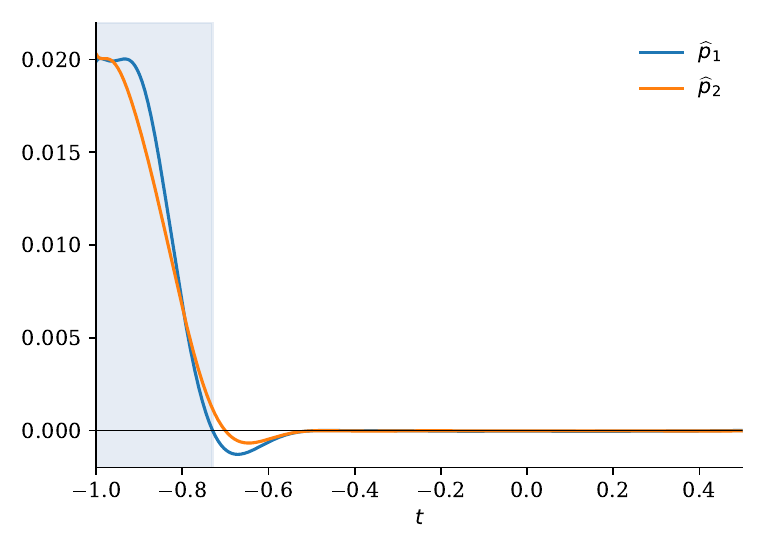}
\end{minipage}\hfill
\begin{minipage}[b]{.47\textwidth}
\includegraphics[width=\linewidth]{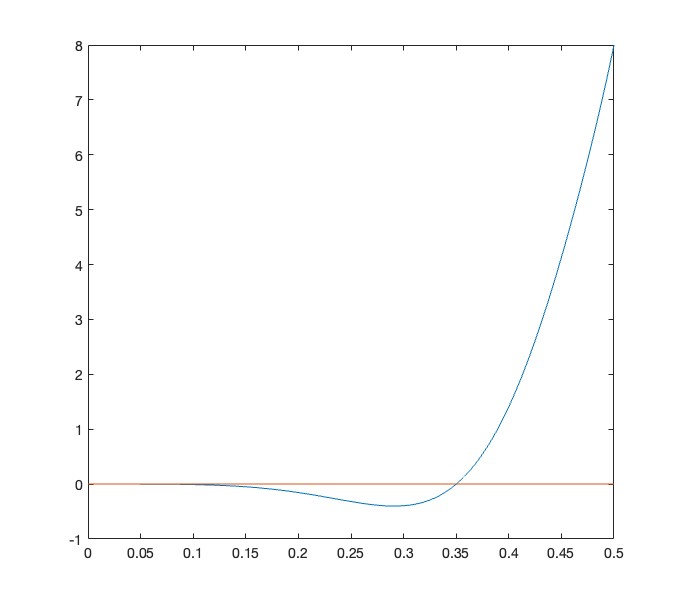}
\end{minipage}
\caption{Left: the two rounded degree-$22$ polynomials, reconstructed
from Appendix~\ref{app:coefficients}; the vertical range emphasizes
their upper levels. Right: the original plot for the third interval
search from the joint working materials \cite{uniq24}. Its original
scale is retained. The figures illustrate the search objectives; they do
not verify the inequalities of the original SDP solutions.}
\label{fig:three}
\end{figure}

\subsection{The interval bounds sought in the experiments}

Table~\ref{tab:intervals} collects the six interval bounds presented
in the joint working materials \cite{uniq24}. They probe the four inner products of the
$24$-cell by grouping neighboring distances. All six are sharp at
that configuration, whose distance distribution is
\begin{equation}\label{eq:D4distribution}
 A_{-1}=1,\qquad A_{-1/2}=8,\qquad A_0=6,\qquad A_{1/2}=8.
\end{equation}
\begin{table}[htbp]
\centering
\caption{Interval bounds from the joint three-point investigation
\cite{uniq24}. The last column is the exact value for the $24$-cell.}
\label{tab:intervals}
\begin{tabular}{lcc}
\toprule
Interval $J$ & Reported bound & $A(J)$ for the $24$-cell\\
\midrule
$[-1,-0.45]$ & $A(J)\le9$ & $9$\\
$[-1,0.35]$ & $A(J)\le15$ & $15$\\
$[-0.73,0.35]$ & $A(J)\le14$ & $14$\\
$[-0.05,0.5]$ & $A(J)\le14$ & $14$\\
$[-1,-0.73]$ & $A(J)\ge1$ & $1$\\
$[0.35,0.5]$ & $A(J)\ge8$ & $8$\\
\bottomrule
\end{tabular}
\end{table}

All six inequalities now follow from the uniqueness theorem of
\cite{DLLM}: every admissible $24$-point code has distribution
\eqref{eq:D4distribution}. Their role here is methodological. They
show which interval counts the earlier three-point searches were
designed to control, rather than providing another proof of uniqueness.
The examples use the retained numerical results of that joint work;
we do not repeat the original SDP computations.

\subsection{A local test for the kissing number}

The joint working materials \cite{uniq24} describe another use of
$p_2$. The reported sign condition is $p_2(t)\le0$ for
$t\in[-1/\sqrt2,1/2]$. For fixed $x\in C$, only neighbors in the open
cap of angular radius $\pi/4$ about $-x$ can then contribute positively.
There are at most four such neighbors in dimension four; see also \cite{BM}
for spherical cap bounds.
Indeed, write them as $y_i=\cos\theta_i(-x)+\sin\theta_i u_i$, with
$u_i\in\Sph^2$ and $\theta_i<\pi/4$. Unless the cap contains only
$-x$, the separation condition gives $u_i\cdot u_j<0$.
A set of vectors in $\R^3$ with pairwise negative inner products has
at most four elements: any linear relation with both positive and
negative coefficients would express a vector as two positive sums
with negative mutual inner product, a contradiction; a nullspace of
dimension at least two would contain such a relation.

Let $H_m(p_2)$ be the supremum of
$\sum_{i=1}^m p_2(x\cdot y_i)$ over these cap configurations, with
$y_i\cdot y_j\le1/2$, and put $H_0=0$. Then
\begin{equation}\label{eq:local-global}
 R_{p_2}(C)\le\max_{0\le m\le4}H_m(p_2).
\end{equation}
The earlier calculation reported a maximum of approximately $0.0266$,
attained in the two-neighbor case. The reported three-point lower
bound at $N=25$ is $0.03241448\ldots$. Their separation suggests the
contradiction
\[
 0.03241448\ldots\ \le R_{p_2}(C)\lesssim0.0266.
\]
This was the proposed shorter route to $k(4)=24$: combine a
three-point lower bound with the local method of \cite{mus08a},
without a five-neighbor case. We record it as an illustration of how
univariate weights can connect SDP bounds with geometric arguments.
The numerical values are those reported in the earlier investigation;
the established kissing number theorem is \cite{mus08a}.

\paragraph{Numerical data.}
The coefficients printed in Appendix~\ref{app:coefficients} are rounded,
whereas the objectives in Table~\ref{tab:targets} refer to the original
searches. The three-point examples are used to explain the method;
the exact theorem below uses the independent rational four-point
construction of \cite{DLLM,DLLMdata}. For computational treatment of
exact SDP bounds, see \cite{DLMexact}.

\section{The four-point reduction}\label{sec:four}

\subsection{Symmetry, Gram domains, and coincidence types}

Write $\boldsymbol t=(t_{12},t_{13},t_{14},t_{23},t_{24},t_{34})$,
and let $\Gamma(\boldsymbol t)$ be the symmetric $4\times4$ matrix
with diagonal entries one and off-diagonal entries $t_{ij}$. Put
\[
 D_4^{(n)}(T)=\{\boldsymbol t\in T^6:
    \Gamma(\boldsymbol t)\PSD,\ \rank\Gamma(\boldsymbol t)\le n\}.
\]
The rank condition is automatic for $n\ge4$. For $n=3$ it must be
retained, although imposing an inequality on the larger domain without
the rank condition is a sufficient alternative.

A function $P$ of these six variables is \emph{vertex-symmetric} if it
is invariant under the permutations of the four vertex labels. This
requires the induced $S_4$ action on the six edges, not arbitrary
permutations of the six coordinates. Define
\[
 S_P(C)=\sum_{x_1,x_2,x_3,x_4\in C}
     P\bigl((x_i\cdot x_j)_{i<j}\bigr).
\]
Assume that
\begin{equation}\label{eq:four-positive}
 S_P(C)\ge P_0N^4.
\end{equation}
As in the three-point case, this is the only positivity assumption
used in the counting argument.

For example, it follows from an invariant positive semidefinite
kernel on ordered pairs. Suppose
$K((x_1,x_2),(x_3,x_4))=Q((x_i\cdot x_j)_{i<j})$ and $K$ is positive
semidefinite. Averaging $Q$ over the vertex permutations gives a
vertex-symmetric $\overline Q$ with
\[
 S_{\overline Q}(C)
 =\sum_{(x_1,x_2),(x_3,x_4)\in C^2}
    K((x_1,x_2),(x_3,x_4))\ge0.
\]
Thus $P=P_0+\overline Q$ satisfies \eqref{eq:four-positive}.
\begin{aside}[Other positivity constructions]
The construction in our earlier work \cite{mus14} also gives
\eqref{eq:four-positive}. Fix two points and sum a positive semidefinite
kernel over two further points. Then sum over the fixed pair and
average over permutations of the four points. Either construction
can be used in the theorem below.
\end{aside}

The restrictions corresponding to repeated vertices are
\begin{align*}
 P_{31}(t)&=P(1,1,t,1,t,t),\\
 P_{22}(t)&=P(1,t,t,t,t,1),\\
 P_{211}(t,u,v)&=P(1,t,u,t,u,v).
\end{align*}
In the last line the tuple is $(x,x,y,z)$, with
$t=x\cdot y$, $u=x\cdot z$, and $v=y\cdot z$. Average over the
choice of the repeated point:
\begin{equation}\label{eq:211-average}
 \overline P_{211}(t,u,v)=\frac{
 P_{211}(t,u,v)+P_{211}(t,v,u)+P_{211}(u,v,t)}3.
\end{equation}
Vertex symmetry makes this a symmetric function of $t,u,v$.

\subsection{The general bound}

\begin{thm}[General four-point bound]\label{th:four}
Let $P$ be vertex-symmetric and satisfy \eqref{eq:four-positive} for
an $(N,n,T)$ code $C$. Suppose $f_{31},f_{22},h,g:T\to\R$ satisfy
\begin{align}
 P_{31}(t)&\le f_{31}(t),\quad P_{22}(t)\le f_{22}(t)
       &&(t\in T),\label{eq:four-pairs}\\
 \overline P_{211}(t,u,v)&\le h(t)+h(u)+h(v)
       &&((t,u,v)\in D_3(T)),\label{eq:four-triples}\\
 P(\boldsymbol t)&\le\sum_{i<j}g(t_{ij})
       &&(\boldsymbol t\in D_4^{(n)}(T)).\label{eq:four-distinct}
\end{align}
Then
\begin{equation}\label{eq:four-general}
\begin{split}
 P_0N^4\le{}&NP(\boldsymbol1)+4E_{f_{31}}(C)+3E_{f_{22}}(C)\\
            &+18(N-2)E_h(C)+6(N-2)(N-3)E_g(C).
\end{split}
\end{equation}
Equivalently, the single weight
\begin{equation}\label{eq:four-weight}
 w_N=4f_{31}+3f_{22}+18(N-2)h+6(N-2)(N-3)g
\end{equation}
satisfies
\begin{equation}\label{eq:four-univariate}
 \boxed{R_{w_N}(C)\ge P_0N^3-P(\boldsymbol1).}
\end{equation}
In particular, polynomial majorants give a univariate polynomial $w_N$.
\end{thm}

\begin{proof}
Partition ordered quadruples by their coincidence types
$4,31,22,211,1111$. With primes denoting sums over distinct points,
\[
 S_P=NP(\boldsymbol1)+4E_{P_{31}}+3E_{P_{22}}
       +6\sum_{x,y,z}'\overline P_{211}
       +\sum_{x_1,x_2,x_3,x_4}'P.
\]
The coefficients $4,3,6$ count the singleton positions, the divisions
into two pairs (ordered pairs count each twice), and the repeated
positions, respectively. Averaging over which of three points is
repeated leaves the triple sum unchanged. Each of its three edges
contributes $(N-2)E_h$; each of the six edges of a distinct quadruple
contributes $(N-2)(N-3)E_g$. Apply the assumed majorants and
\eqref{eq:four-positive}. Division by $N$ gives the univariate form.
\end{proof}

The formula remains valid for $N=2,3$: the factors $N-2$ and
$(N-2)(N-3)$ remove the absent coincidence types. No division by either
factor is used.

\subsection{A prescribed polynomial and the cardinality bound}

The next form uses one weight $g$, independently of $N$. The coefficients
$3,4,5/3$ below have a direct meaning: a tuple of type $31$ has three
edges between unequal points, a tuple of type $22$ has four, and a
tuple of type $211$ has five. Averaging the last five edges over its
three underlying points gives $5/3$ times their three-edge sum.

\begin{cor}[Four-point bound for a prescribed weight]\label{cor:four-balanced}
Assume \eqref{eq:four-positive} and \eqref{eq:four-distinct}, and suppose
\begin{align*}
 P_{31}(t)&\le3g(t)+B_{31},\qquad
 P_{22}(t)\le4g(t)+B_{22},\quad t\in T,\\
 \overline P_{211}(t,u,v)&\le\frac53\bigl(g(t)+g(u)+g(v)\bigr)+B_{211},
       \quad (t,u,v)\in D_3(T).
\end{align*}
Then
\begin{equation}\label{eq:four-balanced}
 R_g(C)\ge
 \frac{P_0N^3-P(\boldsymbol1)
 -(N-1)\bigl(4B_{31}+3B_{22}+6(N-2)B_{211}\bigr)}{6N^2}.
\end{equation}
\end{cor}
\begin{proof}
In Theorem~\ref{th:four}, take
$f_{31}=3g+B_{31}$, $f_{22}=4g+B_{22}$, and
$h=5g/3+B_{211}/3$. The coefficient of $E_g$ becomes
\[
 12+12+30(N-2)+6(N-2)(N-3)=6N^2.
\]
The constant terms sum to
$N(N-1)(4B_{31}+3B_{22}+6(N-2)B_{211})$.
Use $E_g=NR_g$ and divide by $6N^3$.
\end{proof}

When the three constants vanish, the result is
\[
 R_g(C)\ge \frac{P_0N}{6}-\frac{P(\boldsymbol1)}{6N^2}.
\]
When $g=0$ and $P_0>0$, it becomes
\begin{equation}\label{eq:four-cardinality}
 P_0N^3\le P(\boldsymbol1)+4(N-1)B_{31}+3(N-1)B_{22}
                 +6(N-1)(N-2)B_{211}.
\end{equation}
This is the four-point cardinality bound from our earlier work
\cite{mus14}. Inequality~\eqref{eq:four-balanced} extends it to a
bound for the distance distribution. The averaged condition for type
$211$ can also be deduced from the stronger condition
$P_{211}(t,u,v)\le2g(t)+2g(u)+g(v)+B_{211}$.

\subsection{The pair contribution of a four-point kernel}

Let $\mathcal I_2$ be the admissible subsets of the sphere of size at
most two, including the empty set. For an orthogonally invariant positive
semidefinite kernel $K$ on $\mathcal I_2$, put
\[
 A_2K(Q)=\sum_{\substack{J_1,J_2\in\mathcal I_2\\J_1\cup J_2=Q}}
                    K(J_1,J_2).
\]
Assume $K(\varnothing,\varnothing)=M$ and
$A_2K(Q)\le-\mathbf1_{|Q|=1}$ for every nonempty admissible $Q$ of
size at most four. Orthogonal invariance defines a pair function $q_2$ by
$A_2K(\{x,y\})=q_2(x\cdot y)$. Write $r=-q_2\ge0$.

\begin{prop}[A bound for the pair sum]\label{prop:pair-kernel}
For every admissible code $C$ of size $N\ge1$,
\begin{equation}\label{eq:pair-kernel}
 0\le R_r(C)\le\frac{2(M-N)}N.
\end{equation}
If $r(t)\ge c\phi(t)\ge0$ on the allowed interval, where $c>0$, then
\begin{equation}\label{eq:pair-compression}
 0\le R_\phi(C)\le\frac{2(M-N)}{cN}.
\end{equation}
In particular, if $N=M$, all occurring off-diagonal inner products
are zeros of $\phi$, and $R_{-\phi}(C)=0$.
\end{prop}
\begin{proof}
The sum of all entries of the kernel matrix indexed by the sets
$J\subseteq C$ of size at most two is nonnegative. Group by
$Q=J_1\cup J_2$. The empty set contributes $M$, and the singleton
terms sum to at most $-N$. For every $Q\subseteq C$ of size three or four, the grouped
contribution is exactly $A_2K(Q)\le0$ by the feasibility assumption.
Their sum is therefore nonpositive. Keeping the pair terms gives
\begin{equation}\label{eq:set-kernel-sum}
\begin{split}
 0&\le\sum_{\substack{J_1,J_2\subseteq C\\|J_1|,|J_2|\le2}}K(J_1,J_2)\\
  &\le M-N+\sum_{\substack{Q\subseteq C\\|Q|=2}}A_2K(Q)
   =M-N-\frac N2 R_r(C).
\end{split}
\end{equation}
The factor $N/2$ converts the unordered pair sum to the normalized
ordered energy. Now use $R_r\ge cR_\phi\ge0$.
\end{proof}

Here the empty set and the singletons contribute to the bound along
with the pairs. The parameter $M$ is an upper bound for $N$;
the normalization is the same as a cardinality bound.
We apply the proposition to the kernel constructed in \cite{DLLM},
which has objective value $M=24$.

\subsection{A degree-six polynomial for the 24-cell}

We now apply the method to the four-point kernel constructed by
de Laat, Leijenhorst, and de Muinck Keizer \cite{DLLM}. Their kernel
is indexed by sets of at most two points. Proposition~\ref{prop:pair-kernel} retains its pair contribution;
the factorization is proved in Appendix~\ref{app:degree-six}.

\begin{samepage}
\begin{cor}[A degree-six distance restriction]\label{cor:degree-six}
Let $C\subset\Sph^3$ have $24$ points and satisfy $x\cdot y\le1/2$
for distinct $x,y\in C$. The polynomial
\begin{equation}\label{eq:g-six}
 \boxed{g_6(t)=-t^2(1+t)(1-2t)(1+2t)^2}
\end{equation}
satisfies
\begin{equation}\label{eq:g-six-bound}
 R_{g_6}(C)\ge0.
\end{equation}
Since $g_6\le0$ on $[-1,1/2]$, equality holds, and every inner product
between distinct code points belongs to
\[
 Z=\{-1,-1/2,0,1/2\}.
\]
Consequently, $C$ is isometric to the normalized $D_4$ root system.
\end{cor}
\end{samepage}

\begin{proof}[Proof idea]
Write $r=-q_2\ge0$, where $q_2$ is the pair polynomial of the kernel
in \cite{DLLM}, with objective value $M=24$. Summing this positive
semidefinite kernel gives $0\le R_r(C)\le2(M-24)/24=0$.
Factoring its pair sum-of-squares expressions, we obtain
$r(t)=-g_6(t)h(t)$, where $h$ is strictly positive on $[-1,1/2]$.
Equivalently, the reduced polynomial satisfies
\[
 0\le R_{g_6}(C)=\frac1{24}\sum_{\substack{x,y\in C\\x\ne y}}
 g_6(x\cdot y)\le0.
\]
Every summand is nonpositive. If even one inner product were outside
$Z$, its summand would be strictly negative, contradicting equality.
Thus every inner product between distinct points is a zero of $g_6$.
The next subsection gives the geometric steps from these zeros to
uniqueness, following \cite{DLLM}.
The exact factorization is given in Appendix~\ref{app:degree-six}.
\end{proof}

\begin{figure}[htbp]
\centering
\includegraphics[width=0.52\textwidth]{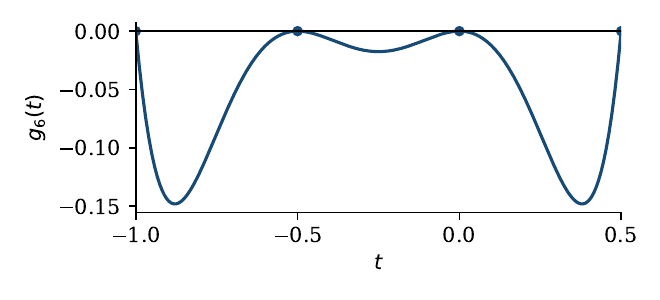}
\caption{The degree-six polynomial is nonpositive on $[-1,1/2]$.
A nonnegative sum over distinct pairs therefore forces every inner
product to be one of its four zeros.}
\label{fig:degree-six}
\end{figure}

The same polynomial, up to a positive scalar factor, appears in the
classical linear programming proof that the kissing number in dimension
eight is $240$, obtained independently by Levenshtein \cite{Lev79}
and Odlyzko and Sloane \cite{OS79} in 1979. The normalized $E_8$ roots have the same
four off-diagonal inner products as the normalized $D_4$ roots.
In dimension eight, all Gegenbauer coefficients of $g_6$ are positive,
its constant coefficient is $3/40$, and $g_6(1)=18$. The two-point
bound gives $N\le18/(3/40)=240$, attained by $E_8$.
In dimension four the coefficients are again positive, but the constant
is $5/8$, so the same argument gives only $N\le144/5$.
For $N=24$ it gives $R_{g_6}(C)\ge-3$, rather than the required zero.

The difference can be seen directly in the sixth Gegenbauer moment.
Write $M_j(C)=\sum_{x,y\in C}G_j^{(n)}(x\cdot y)$, including the
diagonal. For $E_8$, the moments $M_1,\ldots,M_6$ vanish.
For the $24$-cell, $M_1,\ldots,M_5$ vanish but $M_6=576/7$;
the coefficient of $G_6^{(4)}$ is $7/8$. Hence
\[
 R_{g_6}(C_{24})=\frac58\,24-18+
       \frac{7}{8\cdot24}\,\frac{576}{7}=-3+3=0.
\]
Thus the polynomial still sums to zero on the $24$-cell. What loses
sharpness is discarding its positive sixth-moment contribution in the
usual two-point estimate. For an arbitrary admissible $24$-point code,
this moment is not known in advance; the four-point bound supplies the
stronger inequality needed to force the four zeros.

Degree six is the smallest possible for a nonzero polynomial
nonpositive on this interval with these four zeros: the two interior
zeros must have even multiplicity, and each endpoint must be a zero.

\subsection{From the four zeros to uniqueness}

The restriction to four inner products is the first step. To obtain
uniqueness we also use maximality: a $25$th point with inner products
at most $1/2$ cannot be added. We give the argument, following
\cite{DLLM}. The same feasible
kernel, by \eqref{eq:pair-kernel}, gives $N\le24$. If $|C|=24$, the
preceding calculation forces all off-diagonal inner products into $Z$.
For $x\in C$, if $-x\notin C$, then $x\cdot y\ge-1/2$ for all
$y\in C\setminus\{x\}$. Hence $-x$ could be added to the code,
contradicting the cardinality bound. Thus $C$ is antipodal.

Next consider reflection in a hyperplane perpendicular to a point
of $C$. For $\alpha,\beta\in C$, the vector
$\beta'=\beta-2(\alpha\cdot\beta)\alpha$ has unit length. For every
$\gamma\in C$, its inner product is
\[
 \beta'\cdot\gamma=\beta\cdot\gamma
       -2(\alpha\cdot\beta)(\alpha\cdot\gamma)\in\tfrac12\mathbb Z,
\]
since all inner products in $C$ belong to
$\{-1,-1/2,0,1/2,1\}$.
If $\beta'\notin C$, none of these inner products is one, so all are
at most $1/2$. Again $\beta'$ could be added, a contradiction.
Thus $C$ is closed under root reflections, and
$2\alpha\cdot\beta\in\mathbb Z$ for every $\alpha,\beta\in C$.
After multiplication by $\sqrt2$, these vectors form a reduced,
simply laced crystallographic root system of rank at most four. In these ranks the irreducible possibilities are
$A_1,A_2,A_3,A_4,D_4$, with $2,6,12,20,24$ roots, respectively.
Without a $D_4$ component each component has fewer than six roots per
rank, so the total is less than $24$. Consequently $C$ is isometric to
\[
 C_{24}=\{(\pm e_i\pm e_j)/\sqrt2:1\le i<j\le4\},
\]
which completes the proof of Corollary~\ref{cor:degree-six}.

This example has the same structure as the interval experiments:
find a lower bound for a univariate sum, then use the sign of its
weight. Here sharpness leaves only zeros and determines the entire
distance distribution. In the three-point experiments the upper levels
on intervals instead give partial information about the distribution.
The two uses are complementary.

\section{Further questions}

The interval formulation suggests studying how low the degree of the
weight can be while keeping a useful multipoint lower bound. The
auxiliary polynomial may have much higher degree. The degree-six
example shows that a substantial reduction is possible when a sharp
solution is available; for nonsharp bounds the corresponding question
is quantitative. One may minimize the upper level on a chosen interval,
or the error permitted outside it, while fixing the lower bound for
$R_g$.

Another question is how several interval inequalities constrain a
distance graph before they determine all distances. Their integer edge
counts can strengthen individually weak estimates. Local cap bounds
provide additional information, as in the experiment above. These
questions apply beyond the $24$-cell, including spherical codes for
which no sharp multipoint bound is presently available.

\appendix
\section{Rounded coefficients of the experimental weights}
\label{app:coefficients}

Table~\ref{tab:coefficients} specifies
$\widehat p_i(t)=\sum_k c_{ik}G_k^{(4)}(t)$ for $i=1,2$, where
$G_k^{(4)}=U_k/(k+1)$ and $U_k$ is the Chebyshev polynomial of the
second kind. The coefficients are taken from the joint working
materials \cite{uniq24}.
The hats indicate that these are the printed rounded values;
the numerical SDP objectives refer to the original weights.

\begin{table}[htbp]
\centering\small
\caption{Rounded Gegenbauer coefficients. The basis is $G_k^{(4)}(1)=1$.}\label{tab:coefficients}
\begin{tabular}{rrr}
\toprule
$k$ & $\widehat p_1$ & $\widehat p_2$\\
\midrule
0 & 0.222 & -0.5438\\
1 & 0.8648 & -2.0024\\
2 & 1.8875 & -3.8887\\
3 & 3.1425 & -5.6414\\
4 & 4.5059 & -6.7025\\
5 & 5.7052 & -6.8508\\
6 & 6.5739 & -6.0698\\
7 & 6.9286 & -4.6566\\
8 & 6.7119 & -3.0047\\
9 & 6.0157 & -1.4686\\
10 & 4.9575 & -0.3226\\
11 & 3.7767 & 0.3704\\
12 & 2.6446 & 0.6521\\
13 & 1.6914 & 0.6486\\
14 & 0.9947 & 0.5104\\
15 & 0.5249 & 0.3361\\
16 & 0.2524 & 0.1911\\
17 & 0.1097 & 0.0963\\
18 & 0.0409 & 0.0411\\
19 & 0.0153 & 0.0157\\
20 & 0.0042 & 0.0056\\
21 & 0.001 & 0.001\\
22 & 0.0002 & 0.0004\\
\bottomrule
\end{tabular}
\end{table}

The supplementary file \texttt{rounded\_coefficients.csv} contains
these coefficients. The accompanying script draws their graphs,
checks the degree-six identities and the $24$-cell distance distribution,
and verifies the exact pair factorization in the next appendix.
It does not rerun the original three-point optimization.

\section{Exact factorization for the degree-six example}\label{app:degree-six}

\subsection{The degree-six factor}

The accompanying dataset \cite{DLLMdata}, in the folder
\texttt{proofs/4\_24}, gives two rational positive definite pair matrices
$X\in\R^{5\times5}$ and $Y\in\R^{6\times6}$ and rational transformation
matrices $B_1\in\R^{9\times5}$ and $B_2\in\R^{8\times6}$.
Let $v_j(t)=(T_0(t),\ldots,T_j(t))^{\mathsf T}$, where
$T_j(\cos\theta)=\cos(j\theta)$ are the Chebyshev polynomials.
The pair constraint is
\begin{equation}\label{eq:pair-sos}
\begin{split}
 r(t)&=(B_1^{\mathsf T}v_8)^{\mathsf T}X(B_1^{\mathsf T}v_8)
       +w(t)(B_2^{\mathsf T}v_7)^{\mathsf T}Y(B_2^{\mathsf T}v_7),\\
 w(t)&=(1+t)(1/2-t).
\end{split}
\end{equation}
With $\ell(t)=t(t+1/2)$, polynomial division gives
\[
 B_1^{\mathsf T}v_8=w\ell a,\qquad B_2^{\mathsf T}v_7=\ell b,
\]
where $a$ has five polynomial entries and $b$ has six. Since
$\phi:=-g_6=8w\ell^2$, we obtain
\begin{equation}\label{eq:pair-factor}
 r(t)=\phi(t)h(t),\qquad
 h(t)=\frac18\left[w(t)a(t)^{\mathsf T}Xa(t)
                          +b(t)^{\mathsf T}Yb(t)\right].
\end{equation}
The remaining factor $h$ has degree $10$. It is strictly positive
throughout $[-1,1/2]$: here $w\ge0$, both matrices are positive definite,
and the reduced basis satisfies
\begin{equation}\label{eq:b-constant}
 b_1(t)+b_2(t)+b_3(t)=2.
\end{equation}
For example, the three entries in this identity are
\begin{align*}
 b_1(t)&=-56+148t+152t^2-400t^3-96t^4+256t^5,\\
 b_2(t)&=16-32t-48t^2+96t^3+32t^4-64t^5,\\
 b_3(t)&=42-116t-104t^2+304t^3+64t^4-192t^5.
\end{align*}
Thus $b$ never vanishes, proving the assertion. In particular, the
roots of the degree-$16$ polynomial $r$ on the admissible interval
are exactly those of $\phi$, with multiplicities $1,2,2,1$ at
$-1,-1/2,0,1/2$.

We can also estimate the constant in \eqref{eq:pair-compression}.
Put
$u=(1/2,1/2,1/2,0,0,0)^{\mathsf T}$.
Since $u^{\mathsf T}b(t)=1$ by \eqref{eq:b-constant},
Cauchy--Schwarz gives
\[
 b(t)^{\mathsf T}Yb(t)\ge\frac1{u^{\mathsf T}Y^{-1}u}=:\mu>0.
\]
A calculation with rational numbers gives
\[
 \frac\mu8=0.00000793166155227720\ldots>10^{-6},
 \qquad r(t)\ge10^{-6}\phi(t)\ge0.
\]
The last strict comparison is checked as a rational inequality;
the decimal is displayed only to indicate its size.
Proposition~\ref{prop:pair-kernel}, with $M=N=24$, now gives
$R_{g_6}(C)=0$ and the distance restriction in
Corollary~\ref{cor:degree-six}.

The Gegenbauer expansion is
\begin{equation}\label{eq:g-six-gegenbauer}
\begin{split}
 g_6={}&\frac58G_0^{(4)}+\frac94G_1^{(4)}+\frac{15}4G_2^{(4)}
       +\frac92G_3^{(4)}\\
       &+\frac{15}4G_4^{(4)}+\frac94G_5^{(4)}+\frac78G_6^{(4)}.
\end{split}
\end{equation}
All coefficients are positive, and $g_6(1)=18$, proving the direct
lower bound $-3$ stated after the corollary.

The supplementary script checks the positive definiteness of $X,Y$
by rational elimination, the divisions in \eqref{eq:pair-factor},
identity~\eqref{eq:b-constant}, the comparison constant, and
\eqref{eq:g-six-gegenbauer}. The six pair-data files used by the script
are copied unchanged from \cite{DLLMdata}. These checks concern the
univariate pair expression; feasibility of the complete four-point
kernel is the result verified in \cite{DLLM,DLLMdata}.

\bigskip
\noindent\textbf{Acknowledgments.}
The author thanks Maria Dostert, Alexander Kolpakov, and Philippe Moustrou
for the joint work on the three-point interval bounds, and Maria Dostert
for carrying out the polynomial searches.

\medskip
\noindent\textbf{Data availability.}
The supplementary material contains the rounded coefficients,
reproducible polynomial calculations, and the pair data used in
Appendix~\ref{app:degree-six}. The full four-point construction is
available in \cite{DLLMdata}. Full auxiliary data for the earlier
three-point searches are not part of the available material.

\medskip
\noindent\textbf{AI use statement.}
Large language models (LLMs) were used to assist with extensive computations,
verification of polynomial identities, and manuscript revision.

\bigskip
\noindent School of Mathematical and Statistical Sciences,\\
The University of Texas Rio Grande Valley, Brownsville, TX, USA\\
\href{mailto:oleg.musin@utrgv.edu}{\texttt{oleg.musin@utrgv.edu}}
\end{document}